\documentclass[letterpaper,12pt]{amsart}
\usepackage{amsmath,amssymb,amsthm,tikz,setspace,relsize,float}
\usetikzlibrary{calc}
\usepackage[top=1.5in, bottom=1in, left=1in, right=1in]{geometry}
\usepackage{todonotes}

\usepackage{chngcntr}
\counterwithout{figure}{section}

\newcommand{\tikzAngleOfLine}{\tikz@AngleOfLine}
  \def\tikz@AngleOfLine(#1)(#2)#3{%
  \pgfmathanglebetweenpoints{%
    \pgfpointanchor{#1}{center}}{%
    \pgfpointanchor{#2}{center}}
  \pgfmathsetmacro{#3}{\pgfmathresult}%
  }

\newcommand{\lv}{\lvert}
\newcommand{\rv}{\rvert}
\newcommand{\X}{\xi}

\newtheorem{thm}{Theorem}[section]
\newtheorem{lemma}[thm]{Lemma}

\newtheorem{cor}[thm]{Corollary}

\theoremstyle{definition}
\newtheorem{definition}[thm]{Definition}
\newtheorem{obs}[thm]{Observation}
\newtheorem{ques}[thm]{Question}

\newcommand{\zz}{\mathbb{Z}}

\newcommand{\ind}{\mbox{$\perp \kern-5.5pt \perp$}}

\title[Random Walks Determined by a Linear Recurrence]{On Mixing Behavior of a Family 
  of Random Walks Determined by a Linear Recurrence}
\author{Caprice Stanley}
\author{Seth Sullivant}
\email{crstanl2@ncsu.edu}
\email{smsulli2@ncsu.edu}
\address{Department of Mathematics, Box 8205, North Carolina State University, Raleigh, NC, 27695-8205, USA }
\date{}

\begin{document}
\maketitle

\begin{abstract}
We study random walks on the integers mod $G_n$ that are determined by an 
integer sequence $\{ G_n \}_{n \geq 1}$ generated by a linear recurrence relation. 
Fourier analysis provides explicit formulas to compute the eigenvalues of the transition matrices 
and we use this to bound the mixing time of the random walks. 
\end{abstract}


\section{Introduction}

Let $\{G_n \}_{n \geq 1}$ be a positive increasing integer sequence given 
by the linear recurrence with constant coefficients
\[ G_n = \alpha_1G_{n-1} + \alpha_2G_{n-2} + \cdots + \alpha_dG_{n-d},\]
and $G_1=1.$
The sequence determines a family of random walks: 
for fixed $n$, consider the Markov chain $(X_t)_{t\geq 0}$ whose state 
space is $\mathcal S = \mathbb Z_{G_n}.$ The initial state is $X_0=0$ and 
from the current state $X_t$, the next state is 
\[
 X_{t+1}\equiv X_t + z_t \mod G_n, 
\]
where $z_t$ is chosen from the set $\mathcal M = \{ G_1, G_2, \ldots, G_n\} $ 
uniformly at random. So for each $n$ we have an associated Markov chain, 
more specifically, a random walk on the finite abelian group $(\mathbb Z_{G_n},+).$  
By the assumption $G_1=1$, the set $\mathcal M$ generates the group and hence the 
random walk is irreducible. Further as $G_n \in \mathcal M$, the walk is aperiodic. 
The stationary distribution $\vec{ \pi}$, 
to which the random walk converges, is uniform over $\mathcal S$.

This paper examines the number of steps required for the distribution of $X_t$
to be close to its stationary distribution.  It is well-known that this number, 
or mixing time, is governed by the second largest eigenvalue modulus (SLEM) and, 
in general, related to the collection of nontrivial eigenvalues of the transition matrix. 
In the next section, we formalize the notion of mixing time and make the relationship 
between that and eigenvalues concrete. We also introduce notations and established 
results that will be used throughout this work. Section \ref{ubsection} details explicit 
formulas for the eigenvalues and we prove that for a random walk arising from 
$\{G_n \}_{n\geq 1}$ subject to certain conditions, at most $\kappa n^2$ steps 
will suffice where $\kappa$ is some constant that depends on $\{G_n \}_{n\geq 1}$. 
Section \ref{spcasesection} focuses on random walks arising from first order recurrences. 
In that case we show that $\gamma n \log n $ steps will suffice, 
where $\gamma$ is also some constant that depends on $\{G_n \}_{n\geq 1}$. 

Our results on the eigenvalues of these Markov chains also allow us to derive lower bounds on the
mixing times, in the case that
$G_n$ grows like an exponential function.
For general linear recurrences, we have the lower bound of $\kappa n/ \log n$
and in the first order case we get a lower bound of $n$.

Random walks on the integers modulo an integer have been studied frequently,
as they are a prototypical example of a Markov chain on a group,
and are amenable to techniques based on discrete Fourier analysis.
In his review article \cite{SaloffCoste2004}, Saloff-Coste considers, among other things, 
random walks on  $\mathbb Z_p$ given by the $X_{t+1}\equiv X_t + z_t \mod p$ 
where $\mathbb P(z_t =a) = \mathbb P(z_t =b) = \frac12$ for some choice of 
$a,b \in \mathbb Z_p$. Hildebrand \cite{Hildebrand1994} considers walks on 
$\mathbb Z_p$ given by the $X_{t+1}\equiv X_t + z_t \mod p$ where $z_t$ 
is uniform on a set of $k$ random elements of $\mathbb Z_p$. He shows that 
if $n$ is prime then it suffices to take $\kappa n^{2/(k-1)}$ steps to be 
close to uniformly distributed for almost all choices of $k$ elements. 
Hildebrand also considers the case where the size of the random  step set
grows with $n$, and the situation studied in this paper provides an interesting
deterministic boundary case between Theorems 3 and 4 of  \cite{Hildebrand1994}.
Diaconis \cite{Diaconis1988} discusses various random walks on $\mathbb Z_p$ 
given by the $X_{t+1}\equiv a_tX_t + z_t \mod p$, where $a_t$ and $z_t$ 
are subject to various restrictions. 

Though we have proven these upper and lower bounds on the mixing times, we suspect, from simulations, 
that the mixing time grows like $n$ instead of $n \log n$ or $n^2 \log n$. 

The table below displays the mixing times for random walks arising from three integer sequences.

\vspace{5mm}

\hspace{10mm}
\begin{tabular}{|p{.5cm}|||p{2cm}|p{1cm}|||p{2cm}|p{1cm}|||p{2.5cm}|p{1cm}|}
\hline
 \multicolumn{7}{|c|}{Mixing Times for Three Sequences} \\
 \hline
$n$ & \tiny{$G_n=2^{n-1}$} & $t_{mix}$ & \tiny{$G_n=3^{n-1}$} & $t_{mix}$ & \tiny{$G_n=3G_{n-1}-G_{n-2}$} & $t_{mix}$ \\
\hline
1 & 1 & 0 & 1 & 0 & 1 & 0\\
 \hline
 2 & 2 & 1 & 3 & 2 & 3 & 2 \\
 \hline
 3 & 4 & 2 & 9 & 3 & 8 & 3 \\
 \hline
 4 & 8 & 2 & 27 & 3 & 21 & 3 \\
 \hline
 5 & 16 & 3 & 81 & 4 & 55 & 3\\
 \hline
 6 & 32 & 3 & 243 & 4 & 144 & 4\\
 \hline
 7 & 64 & 3 & 729 & 4 & 377 & 4 \\
 \hline
 8 & 128 & 4 & 2187 & 5 & 987 & 4 \\
 \hline
 9 & 256 & 4 & 6561 & 5 & 2584 & 4 \\
\hline
\end{tabular}
\vspace{5mm}


\section{Preliminary Results}
\label{bkg}

This section collects relevant definitions, notations, and theorems
on linear recurrences, mixing times, and Markov chains on groups that we
will use.
A more detailed study of probability and mixing time related items  
can be found in \cite{Levin2008}, and the importance of group structure
for analyzing eigenvalues of Markov chains appears in \cite{Diaconis1988}.

A standard theorem of elementary combinatorics characterizes the solutions of
linear recurrence relations  (see, e.g.~ \cite[Chapter 4]{Stanley1997}):

\begin{thm}
\label{linrec} The sequence $\{ G_n\}_{n\geq 1}$ satisfies
\[
G_n - \alpha_1G_{n-1} - \alpha_2G_{n-2} - \cdots - \alpha_dG_{n-d}=0 
\]
 exactly when for all $n \geq 0$, 
\[ 
 G_n = \sum_{i=1}^l P_i(n) \gamma_i^n 
 \]
where $1-\alpha_1x-\alpha_2x^2 - \cdots - \alpha_dx^d = \prod_{i=1}^l (1-\gamma_ix)^{d_i}$, the $\gamma_i$'s are distinct and nonzero, and each $P_i(n)$ is a polynomial of degree less than $d_i$.
\end{thm}

One consequence we make frequent use of is that there will exist a $\kappa_1 > 0$ such that
$\log G_n  \leq  \kappa_1 n$ for all $n$.  We say that the sequence $\{G_n\}$ exhibits \emph{exponential growth}
if there exists $\kappa_2 > 0$ such that $\kappa_2 n  \leq \log G_n$ for all sufficiently large $n$.

For a Markov chain with transition matrix $P$ a \emph{stationary distribution} $\pi$ is a distribution that satisfies 
$\pi = \pi P.$   For Markov chains that are irreducible and aperiodic, as
ours is, there exists a unique stationary distribution, and any starting distribution
converges to the unique stationary distribution as $t$ goes to infinity.
One of the main tasks in the analysis of a Markov chain is to describe how close the random walk
$(X_t)_{t\geq 0}$ is to its stationary distribution after $t$ steps. 
For this purpose, it is standard to work with  the \emph{total variation distance} on probability distributions:

\begin{definition}
For probability distributions $\mu$ and $\eta$ on the set $\mathcal S$, \emph{total variation distance} is
\[ 
\| \mu-\eta \|_{TV}=\frac 12 \sum_{x \in \mathcal S} \lv \mu(x)-\eta(x) \rv.
 \]
\end{definition}
If $P_0^t$ is the distribution of $(X_t)_{t\geq 0}$ at time $t$, 
then we will be interested in the value $\| P_0^t-\pi \|_{TV}$. 
Moreover we would like to know when $\| P_0^t-\pi \|_{TV}$
 is small as this will translate to the walk being ``close to stationarity." 

\begin{definition}
The \emph{mixing time} for the random walk $(X_t)_{t\geq 0}$ is 
$t_{mix}(\epsilon) = \min\{ t : \| P_0^t-\pi \| \leq \epsilon \}$. 
By convention we let $t_{mix}=t_{mix}(1/4)$.
\end{definition}

The following lemma is a rephrasing of the Upper Bound Lemma which allows us to use a sum involving the eigenvalues of the transition matrix as an approximation for the distance to stationarity at time $t$. 

\begin{lemma}[Upper Bound Lemma, \cite{Diaconis1981}]
\label{ubl}
Let $P_0^t$ be the $t$-step distribution of the random walk on the finite abelian group $(\mathbb Z_{G_n}, +)$ as described previously and let $\pi$ be the uniform distribution over $\mathbb Z_{G_n}$. Then,
\begin{align*}
\| P_0^t -\pi \|_{TV}^2 \leq \frac 14 \sum_{k=1}^{G_n-1} |\lambda_k |^{2t},
\end{align*}
where $\lambda_k$'s are nontrivial eigenvalues of the transition matrix of the random walk.
\end{lemma}

Lemma \ref{ubl}, combined with bounds on the eigenvalues of the transition matrices can be used
to get upper bounds on the mixing times of random walks over our finite group.
Similarly, lower bounds on the largest nontrivial eigenvalue modulus can give lower bounds
on the mixing time:

\begin{lemma}
\label{relaxlb}

For the random walk on the finite abelian group $(\mathbb Z_{G_n}, +)$ as described previously with transition matrix $P$,
\[ t_{mix}(\epsilon) \geq (\tfrac{1}{1-\lambda_*}-1)\log(\tfrac{1}{2 \epsilon})
\]
where $\lambda_*=\max \{ \lv \lambda \rv : \lambda \text{ is an eigenvalue of } P , \lambda \neq 1 \}$.


\end{lemma}

Lemma \ref{relaxlb} also holds for reversible, irreducible, aperiodic Markov chains, 
a proof of which can be found in \cite{Levin2008}. 
The same proof holds in our case since $P$ has an orthonormal basis 
of eigenfunctions with respect to the standard complex inner product 
$\langle f,g \rangle=\frac{1}{ G_n } \sum_{x\in \mathbb Z_{G_n}}f(x)\overline{g(x)}.$

The transition matrix $P$ for the random walk $(X_t)_{t \geq 0}$
associated to the sequence $\{G_n\}_{n \geq 1}$
is the $G_n \times G_n$ circulant matrix whose $ij$-th entry is 

\[ P_{ij}=
\begin{cases}
\frac 1n &  \text{ if } j-i \mod{G_n} \in \mathcal M \\
0 & \text{ otherwise}.
\end{cases}
\]

Since $P$ is a circulant matrix its eigenvalues $\lambda_1, \ldots, \lambda_{G_n}$ can be written explicitly. Let $\X_{G_n}= \exp(\frac{2\pi i}{ G_n})$ be a primitive $G_n$-th root of unity, then 
\begin{align} 
\lambda_k=\frac1n \sum_{i=1}^{n} \X_{G_n}^{kG_i} & \phantom{ee} \text{ for } k=1,2, \ldots G_n.
\end{align}
This formula for the eigenvalues plays an important role in the results that follow.

\section{General Linear Recurrences}
\label{ubsection}

In this section, we prove bounds on nontrivial eigenvalue moduli for linear recurrence relations of arbitrary order.  From this we are
able to deduce lower and upper bounds on the mixing time of the Markov
chain.  In the next section, we specialize to the case of first order linear
recurrences, where we are able to prove stronger upper and lower bounds.

The main result of this section is the following:

\begin{thm} \label{gencaseup} 
For the random walk determined by the linear recurrence $\{ G_n\}_{n\geq 1}$ with $G_1=1$,
the mixing time satisfies:
\[ t_{mix}(\epsilon) \leq \kappa n \log(G_n-1)- \kappa n \log(4\epsilon^2), \phantom{e} \text{ where } \kappa=\frac{1}{4-4\cos(\frac{\pi}{s+1})}. \] 
\end{thm}

Note that for large $n$, there is a constant $\kappa_1$ such that
$\log(G_{n}-1)  \leq \kappa_1 n$.  So from this bound we have the 
following corollary.

\begin{cor}
For the random walk determined by the linear recurrence $\{ G_n\}_{n\geq 1}$ with $G_1=1$,
$t_{mix}  \leq  \gamma n^2$ for some $\gamma$.  
\end{cor}

The overall strategy to prove Theorem \ref{gencaseup} is to bound the modulus of the eigenvalues 
of the transition matrix and then appeal to Lemma \ref{ubl}. We first establish a few lemmas. 

\begin{lemma}
\label{angleslemma}
Let $a >0$ be some real number. If $\theta \in [\frac{2\pi}{a+1}, \frac{2\pi a}{a+1}]$ then 
\[
\lv 1 + \exp(\theta i) \rv \leq \lv 1 + \exp(\tfrac{2\pi i}{a+1}) \rv.
\] 
\end{lemma}

\begin{proof} If $\theta \in [\frac{2\pi}{a+1}, \frac{2\pi a}{a+1}]$ then $\cos(\theta) \leq \cos(\frac{2 \pi}{a+1})$ so 
\begin{align*}
\lv 1 + \exp(\theta i) \rv  & = \sqrt{2 + 2\cos(\theta)} \\
 & \leq \sqrt{2 + 2\cos(\tfrac{2 \pi}{a+1})} \\
 & = \lv 1 + \exp(\tfrac{2\pi i}{a+1}) \rv.   \qedhere
\end{align*}
\end{proof}

\noindent 
Now for each $G_i$ we identify a subset $A_i$ of $[0, 2\pi]$. Let
\[ A_i := \bigcup_{m=0}^{G_i-1}  \left[ \frac{2\pi}{(s+1)G_i} + \frac{2\pi m}{G_i},\frac{2\pi s}{(s+1)G_i}+\frac{2\pi m}{G_i} \right], \text{ where } s=\sum_{j:\alpha_j > 0} \alpha_j. \]
Notice that each $A_i$ satisfies the property that if the angle $\frac{2\pi k}{G_n}$ is in $ A_i$, then $\tfrac{2\pi kG_i}{G_n} \mod{2\pi} \in [\tfrac{2\pi}{s+1}, \tfrac{2\pi s}{s+1}].$ 

\begin{lemma} \label{aisinterval} 
If $\mathcal A= \cup_{i=1}^{n-1} A_i$ then $\mathcal A= \left[ \frac{2\pi}{(s+1)G_{n-1}} ,\frac{2\pi ((s+1)G_{n-1}-1)}{(s+1)G_{n-1}}  \right]$.
\end{lemma}

\begin{proof}
First note that $A_1= [ \frac{2\pi}{(s+1)G_1}, \frac{2\pi s}{(s+1)G_1}]$. Now suppose $\cup_{i=1}^m A_i$ is an interval, for some $1 \leq m < n$. Since $G_i \leq G_{i+1}$ and $G_i+1 \leq sG_{i+1}$ for all $i$, then inequalities (\ref{extendfirst}) and (\ref{extendsec}) hold:
\begin{align}
\label{extendfirst}
\tfrac{2\pi}{(s+1)G_{i+1}} &\leq \tfrac{2 \pi}{(s+1)G_{i}} \leq \tfrac{2 \pi s}{(s+1)G_{i+1}} 
\leq \tfrac{2\pi s}{(s+1)G_i} 
\end{align}
\begin{align}
\label{extendsec}
 \tfrac{2\pi}{(s+1)G_{i+1}}+\tfrac{2\pi(G_i-1)}{G_i} &\leq \tfrac{2 \pi}{(s+1)G_{i}} +\tfrac{2\pi(G_{i+1}-1)}{G_{i+1}} \leq \tfrac{2 \pi s}{(s+1)G_{i+1}}  +\tfrac{2\pi(G_i-1)}{G_i}
\leq \tfrac{2\pi s}{(s+1)G_i} +\tfrac{2\pi(G_{i+1}-1)}{G_{i+1}}.
\end{align}
It follows that the first and last intervals in the set $A_{m+1}$ extend the endpoints of the interval $\cup_{i=1}^m A_i$.
\end{proof}

\begin{lemma}\label{acontainsall}
 The angle $\frac{2\pi k }{G_n} \mod{2\pi}$ is in $ \mathcal A=\cup_{i=1}^{n-1} A_i$ for each $k=1,2, \ldots, G_n-1$.
\end{lemma}

\begin{proof}
It suffices to show that $[\frac{2\pi}{G_n}, \frac{2\pi(G_n-1)}{G_n}] \subset \mathcal A$. Since $G_n \leq (s+1)G_{n-1}$, then inequality (\ref{acontainsineq}) holds:
\begin{equation}
\label{acontainsineq} \tfrac{2\pi}{(s+1)G_{n-1}} \leq \tfrac{2\pi}{G_n} \leq \tfrac{2\pi(G_n-1)}{G_n} \leq \tfrac{2\pi s}{(s+1)G_{n-1}}+\tfrac{2\pi (G_{n-1}-1)}{G_{n-1}}.  \qedhere
\end{equation}
\end{proof}

\begin{lemma}
\label{eigmodup} For $n\geq 2$ and each $k=1,2, \ldots, G_n-1$, the eigenvalue modulus $\lv \lambda_k \rv$ satisfies the following: 
\[ \lv \lambda_k \rv \leq 1-\tfrac2n(1-\lv \cos(\tfrac{\pi}{s+1}) \rv) \text{ where } s=\sum_{j: \alpha_j > 0} \alpha_j.\]
\end{lemma}

\begin{proof}
We will show that for each $k$ there exists some $j \in \{ 1,2, \ldots, n-1\}$ such that
\begin{align}
\label{cosup} \lv \X_{G_n}^{kG_j} + 1 \rv \leq \sqrt{2 + 2\cos(2 \pi / s+1)}.
\end{align}
Then assuming (\ref{cosup}) holds it follows that 
\begin{align*}
\lv \lambda_k \rv &= \tfrac1n \lv \sum_{i=1}^{n} \X_{G_n}^{kG_i} \rv \\
& \leq \tfrac1n \left( \lv \X_{G_n}^{kG_j} + \X_{G_n}^{kG_n} \rv + \sum_{i\neq j, n} \lv \X_{G_n}^{kG_i} \rv \right)\\
& \leq \tfrac1n \left( n-2+ \sqrt{2 + 2\cos(\tfrac{2 \pi}{ s+1}}) \right) \\
& = 1-\tfrac2n \left(1-\lv \cos(\tfrac{\pi}{s+1}) \rv \right).
\end{align*}

Thus it only remains to show that (\ref{cosup}) holds. By Lemma \ref{angleslemma} it suffices to show that there exists some $j \in \{ 1, 2, \ldots, n-1\}$ such that $\frac{2\pi k G_j}{G_n} \mod{2\pi}$ is in the interval $ [\frac{2\pi}{s+1}, \frac{2\pi s}{s+1}]$. By Lemma \ref{acontainsall}, the angle $\frac{2\pi k}{G_n} \mod{2\pi}$ is in $ \mathcal A$ therefore we can let $j$ be the integer such that $1 \leq j < n$ and $\frac{2 \pi k}{G_n}$ is in $ \mathcal A_j$. Then we have $\frac{2\pi k G_j}{G_n} \mod{2\pi} \in  [\frac{2\pi}{s+1}, \frac{2\pi s}{s+1}]$ and hence
$\lv \X_{G_n}^{kG_j} + 1 \rv \leq \sqrt{2 + 2\cos(\tfrac{2 \pi}{s+1})}.$
\end{proof}

We now prove Theorem \ref{gencaseup}.

\begin{proof}[Proof of Theorem \ref{gencaseup}] 
By Lemma \ref{ubl}, the distance to stationarity after $t$ steps is less than $\epsilon$ when $ \sum_{k=1}^{G_n-1} |\lambda_k |^{2t} \leq 4\epsilon^2. $
If $\kappa=\frac{1}{4-4\cos(\frac{\pi}{s+1})}$  then by Lemma \ref{eigmodup},
\begin{align}
\label{gencaseububs}
\sum_{k=1}^{G_n-1} \lv \lambda_k \rv^{2t} \leq \sum_{k=1}^{G_n-1} (1-\tfrac{1}{2\kappa n})^{2t} \leq (G_n-1) \exp(-\tfrac{t}{\kappa n}).
\end{align}
Notice the right hand side of (\ref{gencaseububs}) is bounded above by $4 \epsilon^2$ when 
$ t\geq n\kappa \log(  \tfrac{G_n-1}{4\epsilon^2} ) $.
\end{proof}

To conclude this section, we prove a lower bound for $t_{mix}$ in the case of general 
linear recurrences where $\{G_n\}$ satisfies the exponential growth condition.

\begin{thm}
\label{gencaselow} 
For the random walk determined by the linear recurrence $\{ G_n \}_{n\geq 1}$ with $G_1=1,$
satisfying the exponential growth condition,  if $n>1$
\[
t_{mix}(\epsilon) \geq \frac {n-\gamma \log n}{ \gamma \log n} \log(\tfrac{1}{2 \epsilon}) 
\]
where $\gamma$ is some constant.
\end{thm}

\begin{proof}
We will show that $\lambda_*$ satisfies the inequality $\lambda_* \geq1- \frac{\gamma \log n}{n}$ then appeal to Lemma \ref{relaxlb}.

Let $m : \mathbb N \rightarrow \mathbb N \cup \{ 0\}$ be the function
\[
m(n) = 
\begin{cases}
\max_{j \in \{1, \ldots, n-1 \}} \{ \frac{G_{n-j}}{G_n} > \frac 1n \}  & \text{ if } \frac{G_{n-1}}{G_n} > \frac 1n \\
0 & \text{ otherwise }
\end{cases}
\]
Recall that one of the eigenvalues $\lambda_1$ has the form:
\[
\lambda_1 =  \frac{1}{n} \sum_{i = 1}^n \xi_{G_n}^{G_i}.
\]
We will use the function $m(n)$ to give a lower bound on $|\lambda_1|$.  The modulus of $\lambda_1$
is bounded from below by the real part of $\lambda_1$.  This real part is
\[
\sum_{i = 1}^n \cos\left( \frac{2 \pi G_i}{G_n}\right).
\]
We can bound this sum from below to see that
 \[ 
 \lv \lambda_1 \rv \geq \frac{1 + (n-m(n)-1)\cos(\frac{2 \pi}{n})-m(n)}{n}
 \]
by replacing all summands $\cos\left( \frac{2 \pi G_i}{G_n}\right)$ where $G_i/G_n < 1/n$ by $\cos(\frac{2 \pi}{n})$
and replacing all summands where $G_i/G_n > 1/n$ by $-1$.  

Further, since $\cos(x) \geq 1-\frac{x^2}{2}$, it follows that
\begin{align*} \lv \lambda_1 \rv &\geq 1-\frac{2m(n)}{n}-\frac{2 \pi^2}{n^2} + \frac{2\pi^2(m(n)+1)}{n^3} \\
&\geq 1-\frac{2m(n)}{n}-\frac{2 \pi^2}{n^2}.
\end{align*}

Now let $\eta_1, \eta_2 >1$ be constants and $p$ be a polynomial 
such that $\eta_1^np(n) \leq G_n \leq \eta_2^n p(n)$ for all $n$. 
Then we observe that $\frac{G_{n-j}}{G_n} > \frac 1n$ holds when the inequality 
$ \frac{\eta_1^{(n-j)}p(n-j)} {\eta_2^{n}p(n)}  \geq \frac 1n $  holds.

By rearranging, this occurs when
\begin{align} j & < \frac{\log n}{\log \eta_1} + \frac{n (\log \eta_1 - \log \eta_2) }{\log \eta_1} + \log \left(\frac{p(n-j)}{p(n)} \right) \\
& \leq \frac{\log n}{\log \eta_1}
\end{align}
(since the two dropped terms are negative).
It follows that $m(n) \leq \frac{\log n}{\log \eta_1}$ and so

\begin{align*}
\lv \lambda_1 \rv & \geq 1-\frac{2 \log n}{n \log \eta_1 }-\frac{2 \pi^2}{n^2} \\
& \geq  1 - \frac{\log n}{n}(  \tfrac{2}{\log \eta_1} +  \tfrac{2 \pi^2}{n \log n})
\end{align*}
For $n \geq 2$, the term $\tfrac{2 \pi^2}{n \log n}$ is bounded above by $\tfrac{\pi^2}{\log 2}$.
\[
\lv \lambda_1 \rv   \geq 1- \frac{\log n}{n}(\tfrac{2}{\log \eta_1}+\tfrac{\pi^2}{\log 2}).
\]

This shows that $\lambda_* \geq 1- \frac{\gamma \log n}{n}$ where $\gamma = \tfrac{2}{\log \eta_1}+\tfrac{\pi^2}{\log 2}.$ Then by Lemma \ref{relaxlb},
$ t_{mix}(\epsilon) \geq \frac{n- \gamma \log n}{\gamma \log n} \log \left(\frac{1}{2 \epsilon} \right).
$
\end{proof}


\section{First Order Recurrences}
\label{spcasesection}

This section considers sequences generated by first order 
recurrences $G_n = c G_{n-1}$, that is, geometric series of the form $1, c, c^2, c^3, \ldots$, 
where $c>1$ is a positive integer.
For these sequences, we show that the order of the mixing time 
of associated family of random walks is between $n$ and $n \log n$.  
The main result of this section is the following upper bound on
mixing time:

\begin{thm} \label{spcase} 
For the random walk determined by the sequence $\{ c^{n-1}\}_{n\geq 1}$, where $c>1$ is an integer,
\[
 t_{mix}(\epsilon)  \leq \kappa n \log ( (n-1)(c-1)) -  \kappa n \log(\log (4\epsilon^2+1)), \phantom{e} \text{ where } \kappa=\tfrac{1}{1-\cos(\pi/c)}.
\]
\end{thm}

\noindent The easier lower bound will be proven in 
Theorem \ref{lbspcase} at the end of the section.
The key to proving Theorem \ref{spcase}  will be to exploit the following relationship 
between the eigenvalues of the $n$-th random walk and the $(n+1)$-th random walk.  
Let $\tilde{\lambda}_{n,k}$ denote the $k$-th unnormalized eigenvalue of 
the $n$-th random walk determined by $\{ c^{n-1}\}_{n\geq 1}$. That is,
\[ 
\tilde{\lambda}_{n,k} =\sum_{i=1}^{n} \X_{c^{n-1}}^{kc^{i-1}} = \sum_{i=0}^{n-1} \X_{c^{i}}^k.
\] 

\begin{obs}
\label{polygonobs}
For each $k=1,2, \ldots, G_n$, we ``lift'' the unnormalized eigenvalue 
$\tilde{\lambda}_{n,k}$  to the set 
\[
\mathcal L_{n,k} = \{ \tilde{\lambda}_{n+1,k+jc^{n-1}} : j=0,1,\ldots, c-1 \}\]
of $c$ unnormalized eigenvalues in the  $(n+1)$-th random walk. 
Each element of $\mathcal L_{n,k}$ is equal to  $\tilde{\lambda}_{n,k}$ plus some value 
of the form $\X_{c^n}^{N}$. That is, 
\begin{align*}
\label{genform}
\tilde{\lambda}_{n+1,k+jc^{n-1}} & = \sum_{i=0}^{n} \X_{c^{i}}^{k+jc^{n-1}}= \tilde{\lambda}_{n,k} + \X_{c^n}^{k+jc^{n-1}}. 
\end{align*}
\end{obs}

Over the course of the next two lemmas, we use Observation \ref{polygonobs} 
and show that each $\lv \tilde{\lambda}_{n,k} \rv$ is bounded above 
by a value of the form $n+\frac m2(1-\cos(\frac{\pi}{c}))$, for some 
$m \in \{0,1, \ldots, n-1 \}$. Once that is established, 
to prove Theorem \ref{spcase} we will apply the Upper Bound Lemma.

\begin{lemma}
\label{addvertices}
Let $c>1$ be an integer, $z \in \mathbb C$, and define sets $\mathcal A $ and $\mathcal B$ as follows:
\begin{align*} \mathcal A &=\{ \lv z + \exp(\tfrac{2\pi ji}{ c}) \rv : j=0,1, \ldots, c-1\} \\
 \mathcal B & = \{ \lv z \rv +1 \} \cup \{ \sqrt{ \lv z \rv ^2 + 2 \lv z \rv \cos ( \tfrac{(2j-1)\pi}{ c}) +1} : j = 1,2, \ldots, \lfloor \tfrac c2 \rfloor \}
 \end{align*} There exists a function $f: \mathcal A \rightarrow \mathcal B$ such that $x \leq f(x)$ for all $x \in \mathcal A$.
\end{lemma}

\begin{proof}
Let $\alpha$ be the angle between $z$ and the vector nearest to $z$ from the set $\{ \exp(\frac{2\pi ji}{ c}) : j=0,1, \ldots, c-1\}$ in the complex plane. So $\alpha$ satisfies the inequality $0 \leq \alpha \leq \frac {\pi}{c}$. We illustrate an example in Figure \ref{generaladdvertbound}.

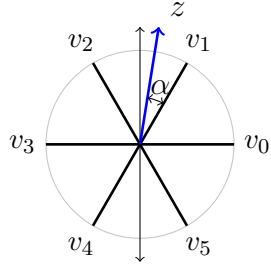
\begin{figure}[H]
\begin{center}
\begin{tikzpicture}[scale=1.25]
\draw [color=gray!45, solid] (0,0) circle (1);

\foreach \ang in {0,1,2,3,4,5} {
  \draw [ line width=1pt] (0,0) -- +(\ang * 360 / 6:1);}
  
  \foreach \ang in {1,3} {
  \draw [->] (0,0) -- +(\ang * 360 / 4:1.25);}

\foreach \ang in {0,1,2,3,4,5} {
  \node at (\ang * 360 / 6:1.25) {\small $v_{\ang}$};}

\draw[blue,line width=1pt,->] (0,0) -- (0.2,1.25);
\node at (0.2,1.25) [above right] {$z$};
\coordinate (A) at (0,0);
\coordinate (B) at (0.5,0.866);
\coordinate (C) at (0.1579,0.9874);

\tikzAngleOfLine(A)(B){\AngleStart}
    \tikzAngleOfLine(A)(C){\AngleEnd}
    \draw[black,<->] (A)+(\AngleStart:0.5cm) arc (\AngleStart:\AngleEnd:0.5 cm);
    \node[circle] at ($(A)+({(\AngleStart+\AngleEnd)/2}:0.65 cm)$) {$\alpha$};

\end{tikzpicture} 
\caption{ \label{generaladdvertbound} Suppose $c=6$, $v_j=\exp(\pi j i/3)$, and $z \in \mathbb C$ as shown. Then $\alpha$ is the angle between $z$ and $v_1$. Lemma \ref{addvertices} gives an upper bound on $\lv z+ v_j \rv$ for each $j$. }
\end{center}
\end{figure}

When $c$ is even, $$\mathcal A=\{\ \sqrt{ \lv z \rv ^2 \pm 2 \lv z \rv \cos (\alpha) +1}  \} \cup \{ \sqrt{ \lv z \rv ^2 \pm 2 \lv z \rv \cos (\tfrac{2j \pi}{c} \pm \alpha) +1} : j = 1,2, \ldots, \tfrac c2 -1 \}.$$ We define the function $f$ as follows:

{\small
\[ f(x)=  \begin{cases} 
      \lv z \rv + 1  & \text{ if } x= \sqrt{ \lv z \rv ^2 + 2 \lv z \rv \cos (\alpha) +1}, \\
      \sqrt{ \lv z \rv ^2 + 2 \lv z \rv \cos (\pi-\tfrac{\pi}{c}) +1} & \text{ if }  x= \sqrt{ \lv z \rv ^2 - 2 \lv z \rv \cos (\alpha) +1},\\
      \sqrt{ \lv z \rv ^2 + 2 \lv z \rv \cos (\tfrac{(2j-1)\pi}{c}) +1} & \text{ if }  x=  \sqrt{ \lv z \rv ^2 \pm 2 \lv z \rv \cos (\tfrac{2j \pi}{c} \pm \alpha) +1}, \text{ for }1 \leq j \leq \tfrac c2 -1.
   \end{cases}
\]
}

It is clear that $f(\mathcal A) \subset \mathcal B$. Now to check that $x \leq f(x)$ for each $x \in \mathcal A$ we consider the three cases. First, since $0 \leq \alpha \leq \pi$, then $$\sqrt{ \lv z \rv ^2 + 2 \lv z \rv \cos (\alpha) +1} \leq \lv z \rv +1.$$ Second, since $\pi -\alpha \geq \pi - \frac{\pi}{c}$, then $-\cos(\alpha)=\cos(\pi -\alpha) \leq \cos(\pi - \frac{\pi}{c})$. Hence, $$\sqrt{ \lv z \rv ^2 - 2 \lv z \rv \cos (\alpha) +1} \leq \sqrt{ \lv z \rv ^2 + 2 \lv z \rv \cos (\pi-\tfrac{\pi}{c}) +1}.$$ Third, for each $j=1,2, \ldots, \frac c2 -1$, the inequality $\frac{2j \pi}{c} \pm \alpha \geq \frac{(2j-1)\pi}{c}$ holds. Hence,
$$ \sqrt{ \lv z \rv ^2 \pm 2 \lv z \rv \cos (\tfrac{2j \pi}{c} \pm \alpha) +1} \leq   \sqrt{ \lv z \rv ^2 + 2 \lv z \rv \cos (\tfrac{(2j-1)\pi}{c}) +1}. $$

When $c$ is odd, 
$$\mathcal A=\{\ \sqrt{ \lv z \rv ^2 + 2 \lv z \rv \cos (\alpha) +1}  \} \cup \{ \sqrt{ \lv z \rv ^2 \pm 2 \lv z \rv \cos (\tfrac{2j \pi}{c} \pm \alpha) +1} : j = 1,2, \ldots, \tfrac{c-1}{2} \}.$$ In this case we define the function $f$ as

{\small
\[ f(x)=  \begin{cases} 
      \lv z \rv + 1  & \text{ if } x= \sqrt{ \lv z \rv ^2 + 2 \lv z \rv \cos (\alpha) +1}, \\
      \sqrt{ \lv z \rv ^2 + 2 \lv z \rv \cos (\tfrac{(2j-1)\pi}{c}) +1} & \text{ if }  x=  \sqrt{ \lv z \rv ^2 \pm 2 \lv z \rv \cos (\tfrac{2j \pi}{c} \pm \alpha) +1}, \text{ for }1 \leq j \leq \tfrac{c-1}{2}.
   \end{cases}
\]
}
By the same arguments used in the even case, $x \leq f(x)$ for all $x \in \mathcal A$.
\end{proof}

Notice that Lemma \ref{addvertices} still holds when we instead define $\mathcal A = \{ \lv z + \exp(\tfrac{2\pi (j+l)i}{ c}) \rv : j=0,1, \ldots, c-1\}$, for some fixed integer $l >0$, since this change corresponds to rotating each $v \in \{  \exp(\tfrac{2\pi ji}{ c}) : j=0,1, \ldots, c-1 \}$ about the origin through the same fixed angle.

\begin{lemma}
\label{seq2bound} 
For $n>1$, define the sets $\mathcal U_n$ and $\mathcal V_n$ as follows:
\begin{align*}
\mathcal U_n &= \{ \lv \tilde{\lambda}_{n,k} \rv : k= 1,2, \ldots, c^{n-1} \} \\
\mathcal V_n &= \{ n+\tfrac m2( \cos(\tfrac{\pi}{c}) - 1) :  m= 0, 1, \ldots, n-1 \}
\end{align*}
 There exists a function $h_n: \mathcal U_n \rightarrow \mathcal V_n$ such that,
\begin{enumerate}
\item $u \leq h_n(u)$ for all $u \in \mathcal U_n$, and 
\item $\#h_n^{-1}(n+\frac m2( \cos(\frac{\pi}{c}) - 1))= \binom{n-1}{m}(c-1)^m$, for $m= 0, 1, \ldots, n-1$.
\end{enumerate}
\end{lemma}

\begin{proof}
Here we use induction. Let $n=2$.  By observation \ref{polygonobs}, the set $\mathcal U_2$ is $\{ \lv \tilde{\lambda}_{1,1} + \X_c^{1+j} \rv : j =0,1, \ldots,c-1 \}$. Note that $\tilde{\lambda}_{1,1}=1$ and 
\[
\{ \X_c^{1+j} : j= 0,1, \ldots, c-1\}= \{ \exp(\frac{2\pi j i}{c}) : j=0,1, \ldots, c-1 \}.
\]
So we can let 
\[
f: \mathcal U_2 \rightarrow \{ 2\} \cup \{ \sqrt{2 + 2\cos(\tfrac{(2j-1)\pi}{c})} : 
j =1,2, \ldots, \lfloor \tfrac c2 \rfloor \}
\]
be as described in proof of Lemma \ref{addvertices} where $u \leq f(u)$ for all $u \in \mathcal U_2$ and define $h_2$ as follows:

\[ 
h_2(u)=
\begin{cases}
2 & \text{if } u \in f^{-1}(2) \\
2+\frac 12(\cos(\tfrac{\pi}{c})-1) & \text{otherwise.}
\end{cases}
\]
 
Since $ \#f^{-1}(2)=1$, then $\#h_2^{-1}(2)=1$ and 
$\# h_2^{-1}(2+\frac 12(\cos(\frac{\pi}{c})-1))=c-1$, so $h_2$ satisfies condition (2). 
For $u \in h_2^{-1}(2)$, the inequality $u \leq h_2(u)$ holds by the triangle inequality. 
If $u \in h_2^{-1}(2+\frac 12(\cos(\frac{\pi}{c})-1))$, then 
$u= \lv \tilde{\lambda}_{1,1} + \X_c^{1+j} \rv $ for some $j$ such that the 
angle between $\tilde{\lambda}_{1,1}$ and $\X_c^{1+j}$, 
when plotted in the complex plane, is greater than or equal to 
$\frac{\pi}{c}$. As a consequence of Lemma \ref{addvertices}, $u \leq \sqrt{2 + 2\cos(\frac{\pi}{c})}$. Now 
\[ 
2+2\cos(\tfrac{\pi}{c})  \leq (2+\tfrac 12(\cos(\tfrac{\pi}{c})-1))^2
\]
since $\frac 14  (\cos(\frac{\pi}{c}) - 1 )^2 \geq 0$ and hence $h_2$ also satisfies condition (1).

Now suppose the Lemma \ref{seq2bound} holds for some $n>1$. We will define a function 
$$h_{n+1}: \{ \lv \tilde{\lambda}_{n+1,k} \rv : k= 1,2, \ldots, c^n \} \rightarrow \{ n+1 + \tfrac m2 (\cos(\tfrac{\pi}{c})-1) : m=0,1,\ldots, n \}$$ that satisfies conditions (1) and (2) assuming there exists a function 
 $$h_n : \{ \lv \tilde{\lambda}_{n,k} \rv : k= 1,2, \ldots, c^{n-1} \} \rightarrow \{ n+\tfrac m2( \cos(\tfrac{\pi}{c}) - 1) :  m= 0, 1, \ldots, n-1 \}$$ that satisfies those conditions. 
 
 For each $k=1,2, \ldots, c^{n-1}$, let 
\[
U_{n+1,k}=\{   \lv \tilde{\lambda}_{n+1,k+jc^{n-1}} \rv : j= 0,1,\ldots,c-1  \}.
\] 
Then by Observation \ref{polygonobs}, 
\[
U_{n+1,k}=\{ \lv \tilde{\lambda}_{n,k} + \X_{c^n}^{k+jc^{n-1} } \rv: j= 0,1,\ldots,c-1  \}
\]
and $\mathcal U_{n+1}=\cup_{k=1}^{c^{n-1}} U_{n+1,k}$. For each $k$, the set 
\[
\{  \X_{c^n}^{k+jc^{n-1} }: j= 0,1,\ldots,c-1\} = \{ \exp(\tfrac{2 \pi k i}{c^n}) \exp(\tfrac{2 \pi j i }{c}) : j= 0,1,\ldots,c-1\}
\]
is a rotation of the set $\{  \exp(\frac{2 \pi j i}{ c}) : j= 0,1,\ldots,c-1\}$ 
about the origin in the complex plane. So we can let 
$ \lv \tilde{\lambda}_{n,k} + \X_{c^n}^{k+j'c^{n-1} } \rv$ 
be an element of $U_{n+1,k}$ such that the vector nearest 
to $\tilde{\lambda}_{n,k}$ from the set $\{  \X_{c^n}^{k+jc^{n-1} }: j= 0,1,\ldots,c-1\}$ 
is $ \X_{c^n}^{k+j'c^{n-1} } $. 
Now set
\[
h_{n+1}(\lv \tilde{\lambda}_{n,k} + \X_{c^n}^{k+j'c^{n-1} } \rv)=h_n(\lv \tilde{\lambda}_{n,k} \rv) +1
\]
and for the remaining $\lv \tilde{\lambda}_{n,k} + \X_{c^n}^{k+jc^{n-1} } \rv \in U_{n+1,k}$, set 
\[
h_{n+1}(\lv \tilde{\lambda}_{n,k} + \X_{c^n}^{k+jc^{n-1} } \rv)=
h_n(\lv \tilde{\lambda}_{n,k} \rv) + \tfrac 12 (\cos(\tfrac{\pi}{c}) +1 ).\]
By repeating for each $k$, we define $h_{n+1}$ on all of $\mathcal U_{n+1}.$ 

It remains to show that $h_{n+1}$ satisfies conditions (1) and (2). We first show that $u \leq h_{n+1}(u) $ for all $u \in U_{n+1}$: 

For $u \in \mathcal U_{n+1}$, $u= \lv \tilde{\lambda}_{n,k} + \X_{c^n}^{k+jc^{n-1} } \rv$ for some $k\in \{1,2, \ldots, c^{n-1} \}$ and some $j \in \{ 0,1,\ldots, c-1\}$. If $h_{n+1}(u)=h_n(\lv \tilde{\lambda}_{n,k} \lv) +1$, then $u \leq h_{n+1}(u)$ by the triangle inequality. On the other hand suppose $h_{n+1}(u)=h_n(\lv \tilde{\lambda}_{n,k} \lv) +\frac 12 (\cos(\frac{\pi}{c})+1)$ and say $h_n(\lv \tilde{\lambda}_{n,k} \lv)=n+\frac {m'}{2}(\cos(\frac{\pi}{c})-1)$ for some $0 \leq m' \leq n-1$. Then  $\lv \tilde{\lambda}_{n,k} \lv \leq n+\frac {m'}{2}(\cos(\frac{\pi}{c}) -1) $ and $ h_{n+1}(u)=n+1+\frac {m'+1}{2}(\cos(\frac{\pi}{c})-1)$. As a corollary to Lemma \ref{addvertices}, 
\begin{align*}
 u & \leq \sqrt{ \lv \tilde{\lambda}_{n,k} \rv^2  + 2 \lv \tilde{\lambda}_{n,k} \rv  \cos(\tfrac{\pi}{c}) +1} \\
 & \leq \sqrt{ (n+\tfrac {m'}{2}(\cos(\tfrac{\pi}{c}) -1))^2  + 2( n+\tfrac {m'}{2}(\cos(\tfrac{\pi}{c}) -1))  \cos(\tfrac{\pi}{c}) +1} \\
 & \leq n+1+\tfrac {m'+1}{2}(\cos(\tfrac{\pi}{c})-1)
\end{align*}
The last step follows since 
\begin{align*}
m'\cos(\tfrac{\pi}{c}) (\cos(\tfrac{\pi}{c})-1) \leq m'\cos(\tfrac{\pi}{c}) + (n+1)(\cos(\tfrac{\pi}{c})-1) + \tfrac{2m'+1}{4}(\cos(\tfrac{\pi}{c})-1).
\end{align*}

Finally we show that $\#h_{n+1}^{-1}(n+1+\frac m2( \cos(\frac{\pi}{c}) - 1))= \binom{n}{m}(c-1)^m$, for $m= 0, 1, \ldots, n$. By inductive hypothesis $ h_n^{-1}(n+\frac m2( \cos(\frac{\pi}{c}) - 1))= \binom{n-1}{m}(c-1)^m$, for $m= 0, 1, \ldots, n-1$. 

We note that $\#h_{n+1}^{-1}(n+1)=\#h_n^{-1}(n)=1$ and for $m'$ satisfying $1\leq m' \leq n$, 
\begin{align*} \#h_{n+1}^{-1}(n+1+\tfrac m2( \cos(\tfrac{\pi}{c}) - 1)) 
&= \#h_n^{-1}(n+\tfrac m2( \cos(\tfrac{\pi}{c}) - 1)) + \#h_n^{-1}(n+\tfrac {m-1}{2}( \cos(\tfrac{\pi}{c}) - 1))\cdot(c-1)\\
&= \binom{n-1}{m}(c-1)^m+\binom{n-1}{m-1}(c-1)^m \\
& = (c-1)^m \binom{n}{m}
\end{align*}
which concludes the proof.
\end{proof}

\begin{proof}[Proof of Theorem \ref{spcase}]
Recall that $\lambda_k = \frac1n \sum_{i=1}^{n} \X_{c^{n-1}}^{kc^{i-1}}$ is the $k$-th eigenvalue of the $n$-th random walk. So $\lv \lambda_k \rv = \frac1n \lv \tilde{\lambda}_{n,k} \rv$. By Lemma \ref{ubl}, to find $t$ such that $\| P_0^t-\pi \|_{TV} \leq \epsilon$, it suffices to find $t$ such that  $ \sum_{k=1}^{c^{n-1}-1} \lv \lambda_k \rv ^{2t} \leq 4\epsilon^2.$ 

If $\kappa=\frac{1}{1-\cos(\pi/c)}$ then by Lemma \ref{seq2bound} we have 
\begin{align}
\label{bineigup}
\sum_{k=1}^{c^{n-1}-1} \lv \lambda_k \rv ^{2t} = \sum_{k=1}^{c^{n-1}-1} \left(\frac1n \lv \tilde{\lambda}_{n,k} \rv \right)^{2t} \leq  \sum_{m=1}^{n-1} {{n-1}\choose{m}}(c-1)^m \left( 1-\tfrac{m}{2\kappa n} \right)^{2t}
\end{align}

The right hand side of  (\ref{bineigup}) can also be bounded above
\begin{align*} 
\leq \sum_{m=1}^{n-1} {{n-1}\choose{m}} (c-1)^m \exp \left(-\tfrac {t}{\kappa n}\right)^m 
\end{align*}
and by the Binomial Theorem,
\begin{align}
\label{expup}
= \left(1+(c-1)\exp(-\tfrac {t}{\kappa n})\right)^{n-1}-1  \leq \exp \left((c-1)(n-1)\exp(-\tfrac {t}{\kappa n}) \right)-1.
\end{align}
Finally, the right hand side of  (\ref{expup}) $\leq 4\epsilon^2 $ when
\[ t \geq \kappa n \log((n-1)(c-1)) - \kappa n \log( \log (4\epsilon^2+1)). \qedhere \]  
\end{proof}

We conclude this section with a lower bound on mixing time.

\begin{thm}
\label{lbspcase}
For the random walk determined by the sequence $\{ c^{n-1}\}_{n\geq 1}$, where $c>1$ is an integer,
\[
 t_{mix}(\epsilon)  \geq (\gamma n -1) \log(\tfrac{1}{2\epsilon}), \phantom{ee} \text{where } \gamma= \tfrac{1}{1-\cos(2 \pi /c)}.
\]
\end{thm}

\begin{proof}
For fixed $n>1$, the modulus of the $k=c^{n-2}$-th eigenvalue satisfies the inequality 
\[
\lv  \lambda_{c^{n-2}} \rv 
= \tfrac 1n \lv \X_c + n -1 \rv 
 \leq  1 - \tfrac{1-\cos(2 \pi /c)}{n}.
\]
So $\lambda_*=\max \{ \lv \lambda \rv : \lambda \text{ is an eigenvalue of } P , \lambda \neq 1 \} \geq 1 - \frac{1-\cos(2 \pi /c)}{n}$, thus by Lemma \ref{relaxlb},
\[
 t_{mix}(\epsilon)  \geq \left( \tfrac{n}{1-\cos(2 \pi/c)}  -1 \right) \log(\tfrac{1}{2\epsilon}).  \qedhere
\]
\end{proof}

\section{Conclusion}

We have shown that the order of the mixing time of random walks determined by a general 
linear recurrence exhibiting exponential growth is between $n / \log n$ and $n^2$. 
A situation that requires further study is the special case where the 
integer sequence defined by the linear recurrence exhibits polynomial 
growth instead. This occurs when the characteristic equation of the 
recurrence is $(1-x)^d$ for some $d \in \mathbb N.$ For this case, 
the result and proof of Theorem \ref{gencaseup} still holds and the 
corresponding upper bound on mixing time is on the order of $n \log n$.
However based on the computations of certain examples, we expect that 
the true mixing time of these random walks are likely bounded by a function of $\log n$.

Proving 
mixing times results for sequences of polynomial
growth seems to be related to some classic problems
in number theory.  For example, consider the following special case:
\begin{ques}
For fixed $k \in \mathbb N_{>1}$ and $n>1$ ranging, 
describe the mixing behavior of the random walk $(X_t)_{t\geq 0}$ 
with state space $\mathcal S = \mathbb Z_{n^k}$, initial state $X_0=0$, 
and where from the current state $X_t$, the next state is given by
\[
 X_{t+1}\equiv X_t + z^k \mod n^k, 
\]
with $z$ chosen from the set $\mathcal M = \{ 1,2,\ldots, n\} $ 
uniformly at random.
\end{ques}
The Hilbert-Waring theorem \cite{Hilbert1909} (which says that there is a function $g(k)$
such that every nonnegative integer can we written as a sum
of at most $g(k)$ $k$-th powers) guarantees that this 
Markov chain has a bounded diameter for all $n$.
The mixing time of the Markov chain appears to be related to the problem of 
determining the number of ways that a number can be written as the sum of
$l$ $k$-th powers.  This has complicated relations to theta functions.

\section*{Acknowledgments}
Caprice Stanley was partially supported by the US National Science Foundation (DMS 1615660).
Seth Sullivant was partially supported by the US National Science Foundation (DMS 1615660) and
by the David and Lucille Packard Foundation.


\begin{thebibliography}{99}
\bibitem{Diaconis1988}  Diaconis, Persi. Group Representations in Probability and Statistics. Institute of Mathematical Statistics. Hayward, California, 1988

\bibitem{Diaconis1981}  Diaconis, Persi and Shahshahani, Mehrdad.  Generating a random permutation with random transpositions. 
\emph{Z. Wahrscheinlichkeitstheorie Verw. Gebiete} {\bf 57} (1981) 159--179.

\bibitem{Hilbert1909}
Hilbert, David.  Beweis f\"ur die Darstellbarkeit der ganzen Zahlen durch eine feste Anzahl n-ter Potenzen (Waringsches Problem). (German)
Math. Ann. 67 (1909), no. 3, 281--300. 

\bibitem{Hildebrand1994}   Hildebrand, Martin.  Random walks supported on random points of $\zz/n\zz$. \emph{Probability Theory and Related Fields}. {\bf 100}, (1994) issue 2, 191-203.

\bibitem{Levin2008}  Levin, David, Peres, Yuval,  and Wilmer, Elizabeth. \emph{Markov Chain Mixing Times.} American Mathematical Society. Providence, RI., 2008.

\bibitem{SaloffCoste2004}  Saloff-Coste, Laurent.  Random Walks on Finite Groups. 
In: Kesten H. (eds) \emph{Probability on Discrete Structures.}  
Encyclopaedia of Mathematical Sciences (Probability Theory), vol 110. Springer, Berlin, Heidelberg, 2004.

\bibitem{Stanley1997}  Stanley, Richard.  Enumerative Combinatorics, Volume 1, 2nd ed, 
Cambridge University Press. New York, 1997.
\end{thebibliography}
\end{document}